\documentclass[10pt]{article}

\usepackage{graphicx}%
\usepackage{amsmath,amssymb,amsthm,upref,bm}%

\newtheorem{corollary}{Corollary}%
\newtheorem{theorem}{Theorem}%
\newtheorem{proposition}{Proposition}%
\begin{document}

\baselineskip=4.4mm

\makeatletter

\newcommand{\E}{\mathrm{e}\kern0.2pt} 
\newcommand{\D}{\mathrm{d}\kern0.2pt}
\newcommand{\RR}{\mathbb{R}}
\newcommand{\CC}{\mathbb{C}}%
\newcommand{\ii}{\kern0.05em\mathrm{i}\kern0.05em}

\renewcommand{\Re}{\mathrm{Re}} 
\renewcommand{\Im}{\mathrm{Im}}

\def\bottomfraction{0.9}

\title{\bf Harmonicity of a function via \\ harmonicity of its spherical means}

\author{Nikolay Kuznetsov}

\date{}

\maketitle

\vspace{-8mm}

\begin{center}
Laboratory for Mathematical Modelling of Wave Phenomena, \\ Institute for Problems
in Mechanical Engineering, Russian Academy of Sciences, \\ V.O., Bol'shoy pr. 61,
St Petersburg 199178, Russian Federation \\ E-mail address:
nikolay.g.kuznetsov@gmail.com
\end{center}

\begin{abstract}
\noindent It is proved that harmonic functions are characterized by harmonicity of
their spherical means, for which purpose the iterated spherical means are used. The
similar characterization of solutions to the modified Helmholtz equation
(panharmonic functions) is given. Another description of harmonic functions is the
pointwise equality of a function and its iterated mean over an admissible pair of
spheres.
\end{abstract}

\setcounter{equation}{0}

%\vspace{-2mm}

\section{Introduction and the main result}

A function $u \in C^2 (D)$ is called harmonic (see \cite{ABR}, p.~25, for the origin
of this term), if it satisfies the equation $\nabla^2 u (x) = 0$ in a domain $D
\subset \RR^m$, $m \geq 2$; $\nabla = (\partial_1, \dots , \partial_m)$ denotes the
gradient operator, $\partial_i = \partial / \partial x_i$ and $x = (x_1, \dots,
x_m)$ is a point of $\RR^m$. Studies of mean value properties of harmonic functions
date back to the Gauss theorem of the arithmetic mean over a sphere; see \cite{G},
Article~20. Nowadays, its standard formulation is as follows.

\begin{theorem}
Let $u \in C^2 (D)$ be harmonic in a domain $D \subset \RR^m$, $m \geq 2$. Then for
every $x \in D$ the equality $M (x, r, u) = u (x)$ holds for each admissible sphere
$S_r (x)$.
\end{theorem}

Here and below the following notation and terminology are used. The open ball of
radius $r$ centred at $x$ is denoted by $B_r (x) = \{ y : |y-x| < r \}$; the latter
is called admissible with respect to a domain $D$ provided $\overline{B_r (x)}
\subset D$, and $S_r (x) = \partial B_r (x)$ is the corresponding admissible sphere.
If $u \in C^0 (D)$, then its spherical mean value over $S_r (x) \subset D$ is\\[-2mm]
\begin{equation}
M (x, r, u) = \frac{1}{|S_r|} \int_{S_r (x)} u (y) \, \D S_y = \frac{1}{\omega_m}
\int_{S_1 (0)} u (x +r y) \, \D S_y \, , \label{sm}
\end{equation}
where $|S_r| = \omega_m r^{m-1}$ and $\omega_m = 2 \, \pi^{m/2} / \Gamma (m/2)$ is
the total area of the unit sphere (as usual $\Gamma$ stands for the Gamma function),
and $\D S$ is the surface area measure.

An immediate consequence of Theorem 1 involves the domain $D_r \subset D$ with
boundary `parallel' to $\partial D$ at the distance $r > 0$; namely, $D_r = \{ x \in
D : \overline{B_r (x)} \subset D \}$. Thus, $D_r$ is nonempty only when $r$ is less
than the radius of the open ball inscribed into~$D$.

\begin{corollary}
Let $u \in C^2 (D)$ be harmonic in a domain $D \subset \RR^m$, $m \geq 2$. If $D_r$
is nonempty for $r > 0$, then the function $M (\cdot, r, u)$ is harmonic in this
domain.
\end{corollary}

The proof apparently follows by applying $\nabla$ twice to the right-hand side in
\eqref{sm}. In view of the latter assertion, it is natural to investigate whether
$u$ is harmonic in $D$ provided each $M (\cdot, r, u)$ with sufficiently small $r$
is harmonic in $D_r$. The following positive answer to this question is the main
result of this note.

\begin{theorem}
Let $D$ be a bounded domain in $\RR^m$, $m \geq 2$, and let $u \in C^0 (\overline D)
\cap C^2 (D)$ be real-valued. If $M (\cdot, r, u)$ is harmonic in $D_r$ for all $r
\in (0, r_*)$, where $r_*$ is a positive number such that $D_{r_*} \neq \emptyset$,
then $u$ is harmonic in $D$.
\end{theorem}

The author failed to find a result of this kind in the literature; in particular,
there is no mention of anything similar in the extensive survey \cite{NV}.

\vspace{-2mm}

\section{Proof of Theorem 2 and discussion}

Prior to proving Theorem 2, let us consider some properties of the iterated
spherical mean introduced by John; see \cite{J}, p.~78, but the notation used here
is different:\\[-2mm]
\begin{equation}
I (x, r', r, u) = M (x, r', M (\cdot, r, u)) = \frac{1}{\omega_m} \int_{S_1 (0)} M
(x + r' y, r, u) \, \D S_y  \, . \label{im}
\end{equation}
The second equality is a consequence of \eqref{sm}. Since $M (\cdot, r, u)$ is
defined on $D_r$, it is obvious that $I (\cdot, r', r, u)$ is defined on $D_{r'+r}$.
Substituting the expression for $M$, we obtain\\[-2mm]
\begin{equation}
I (x, r', r, u) = \frac{1}{\omega_m^2} \int_{S_1 (0)} \int_{S_1 (0)} u (x + r' y + r
z) \, \D S_z \, \D S_y \, , \label{im'}
\end{equation}
and so it is symmetric in $r'$ and $r$, that is, $I (x, r', r, u) = I (x, r, r',
u)$. Moreover,\\[-2mm]
\[ I (x, 0, r, u) = I (x, r, 0, u) = M (x, r, u) \quad \mbox{and} \quad I (x, 0, 0, 
u) = u (x) .
\]
By virtue of the iterated mean \eqref{im'}, Theorem 2 will be reduced to the
converse of Theorem~1 due to Kellogg \cite{K1}; its modern formulation is as
follows.

\begin{theorem}[Kellogg]
Let $D$ be a bounded domain in $\RR^m$, $m \geq 2$, and let $u \in C^0 (\overline
D)$ be real-valued. If for each $x \in D$ there exists $r (x)$ such that
$\overline{B_{r (x)} (x)} \subset D$ and $M (x, r (x), u) = u (x)$, then $u$ is
harmonic in~$D$.
\end{theorem}

\begin{proof}[Proof of Theorem 2]
It is clear that each $x \in D$ belongs to all $D_r$ with $r < \mathrm{dist} (x,
\partial D) / 2$, where $\mathrm{dist} (x, \partial D)$ is the distance from $x$ to
$\partial D$. Let us fix some $r (x) \in (0, \mathrm{dist} (x, \partial D) / 2)$,
and so $\overline{B_{r (x)} (x)} \subset D_r$ for all described values of~$r$. Since
the mean $M (\cdot, r, u)$ is harmonic in $D_r$ for each of these values, we have
$M (x, r (x), M (\cdot, r, u)) = M (x, r, u)$ by Theorem~1. In view of \eqref{im'}
and \eqref{sm}, this can be written as follows:\\[-2mm]
\[ \frac{1}{\omega_m^2} \int_{S_1 (0)} \int_{S_1 (0)} u (x + r (x) y + r
z) \, \D S_z \, \D S_y = \frac{1}{\omega_m} \int_{S_1 (0)} u (x +r y) \, \D S_y \,.
\]
Letting $r \to 0$ in this equality, we obtain that $M (x, r (x), u) = u (x)$ holds
for each $x \in D$ with some $r (x)$ such that $\overline{B_{r (x)} (x)} \subset D$.
Now, Theorem~3 yields that $u$ is harmonic in~$D$.
\end{proof}

The proof looks simple, but it relies upon Theorem~3 which is not trivial at all.
However, the strong converse of Theorem~1 is easy to prove when $D$ is the so-called
Dirichlet domain; that is, a bounded domain in which the Dirichlet problem for the
Laplace equation is soluble provided the function given on $\partial D$ is
continuous. To illustrate the advantage of Dirichlet do\-mains, let us prove the
assertion (Theorem~4 below) similar to Theorem~3, but involving the iterated
mean~\eqref{im} instead of $M (\cdot, r, u)$. Theorem~4 as well as in the next
proposition require admissible triples instead of admissible spheres, thus allowing
us to consider $I (x, r', r, u)$ for any $x \in D$. Namely, the triple $(x, r', r)$
is admissible with respect to~$D$, if $x + r' y + r z$ belongs to this domain for
all $y, z \in B_1 (0)$.

\begin{proposition}
Let $u \in C^2 (D)$ be harmonic in a domain $D \subset \RR^m$, $m \geq 2$. Then for
every $x \in D$ the equality $I (x, r', r, u) = u (x)$ holds for each admissible
triple $(x, r', r)$.
\end{proposition}

\begin{proof}
Since $u$ is harmonic in $D$, the equality $M (x + r' y, r, u) = u (x + r' y)$ holds
for every $x \in D$ and all $y \in B_1 (0)$ provided the triple $(x, r', r)$ is
admissible. Then the result follows by using this equality in the integral on the
right-hand side of \eqref{im} with subsequent application of Theorem~1 to the
obtained integral.
\end{proof}

Now, let us prove the following strong converse of Proposition 1.

\begin{theorem}
Let $D$ be a Dirichlet domain in $\RR^m$, $m \geq 2$, and let $u \in C^0 (\overline
D) \cap C^2 (D)$ be real-valued. If for every $x \in D$ there exist $r' (x)$ and $r
(x)$ such that the triple $(x, r' (x), r (x))$ is admissible and the equality $I (x,
r' (x), r (x), u) = u (x)$ holds, then $u$ is harmonic in $D$.
\end{theorem}

\begin{proof}
First, let us demonstrate that theorem's assumptions yield that\\[-3mm]
\begin{equation}
\max_{x \in \overline D} u (x) = \max_{x \in \partial D} u (x) \, . \label{max}
\end{equation}
Denoting the left-hand side by $U$, we show that the closed preimage $u^{-1} (U)$
has a nonempty intersection with $\partial D$. Indeed, if $u^{-1} (U) \cap \partial
D = \emptyset$, then there exists $x_0 \in u^{-1} (U) \subset D$ that is nearest to
$\partial D$, and so for some admissible triple $(x_0, r' (x_0), r (x_0))$ we
have:\\[-2mm]
\[ U = u (x_0) = \frac{1}{\omega_m^2} \int_{S_1 (0)} \int_{S_1 (0)} u (x_0 + r' (x_0) 
y + r (x_0) z) \, \D S_z \, \D S_y \, .
\]
In view of the maximality of $U$, the equality $u (x_0 + r' (x_0) y + r (x_0) z) =
U$ holds for all $y, z \in S_1 (0)$, that is, every $x_0 + r' (x_0) y + r (x_0) z$
belongs to $u^{-1} (U)$. Hence the distance to $\partial D$ from some point $x_0 +
r' (x_0) y_0 + r (x_0) z_0 \in D$ with $y_0, z_0 \in S_1 (0)$ is smaller than
from~$x_0$. The obtained contradiction yields \eqref{max}.

Let $f$ denote the trace of $u$ on $\partial D$; then there exists $u_0 \in C^0
(\overline D)$ solving the Dirichlet problem for Laplace equation in $D$ with $f$ as
the boundary data. Therefore, Theorem~1 is valid for $u_0$, and so theorem's
assumptions are fulfilled for $u - u_0$ and $u_0 - u$. Since both these functions
vanish on $\partial D$, equality \eqref{max} yields that $u - u_0 \leq 0$ and $u_0 -
u \leq 0$ in $D$. Thus, $u$ is harmonic in $D$, being equal to $u_0$ there.
\end{proof}

\vspace{-5mm}

\subsection{Extension of Theorem 2 to panharmonic functions}

It was Duffin \cite{D}, who introduced the convenient abbreviation `panharmonic
functions' for the awkward `solutions of the modified Helmholtz equation' which
arise in numerous applications; see \cite{CHL}. Since the most important of them
concerns nuclear forces, the equation\\[-2mm]
\begin{equation}
\nabla^2 u - \mu^2 u = 0 , \quad \mu \in \RR \setminus \{0\} , \label{MHh}
\end{equation} 
is referred to as the Yukawa equation in \cite{D} (surprisingly, without citing the
original paper \cite{Y}, in which Yukawa proposed his theory of these forces). Much
attention has been given~to solving \eqref{MHh} numerically (see \cite{CHL} again),
but no analogue of Theorem~1 for panharmonic functions was proved until recently.
The following assertion about the $m$-dimensional mean for spheres was obtained in
\cite{Ku2}.

\begin{theorem}
Let $u \in C^2 (D)$ be panharmonic in a domain $D \subset \RR^m$, $m \geq 2$. Then
for every $x \in D$ the equality\\[-4mm]
\begin{equation}
M (x, r, u) =  a (\mu r) \, u (x) , \quad a (\mu r) = \Gamma \left( \frac{m}{2}
\right) \frac{I_{(m-2)/2} (\mu r)}{(\mu r / 2)^{(m-2)/2}} \, ,
\label{MM}
\end{equation}
holds for each admissible sphere $S_r (x);$ $I_\nu$ denotes the modified Bessel
function of order $\nu$.
\end{theorem}

For $m=3$ formula \eqref{MM} has particularly simple form because $a (\mu r) = \sinh
\mu r / (\mu r)$, and this was proved by C.~Neumann \cite{NC} as early as 1896.
Duffin independently rediscovered his proof (see \cite{D}, pp.~111-112), but for the
two-dimensional case when $a (\mu r) = I_0 (\mu r)$.

\begin{corollary}
Let $u \in C^2 (D)$ be panharmonic in a domain $D \subset \RR^m$, $m \geq 2$. If for
$r > 0$ the domain $D_r$ is nonempty, then the function $M (\cdot, r, u)$ is
panharmonic in this domain; the coefficient in equation \eqref{MHh} for $M (\cdot,
r, u)$ and $u$ is the same.
\end{corollary}

The proof of this corollary is similar to that of Corollary~1, but along with
equation \eqref{MHh} formula \eqref{MM} must be used. Now, let us turn to the
version of strong converse of Theorem~5 valid for Dirichlet domains which are the
same for harmonic and panharmonic functions; this is a well-known consequence of
results obtained in \cite{O} and \cite{T}. In the paper \cite{Ku2}, the following
assertion was established.

\begin{theorem}
Let $D \subset \RR^m$ be a Dirichlet domain, and let $u \in C^0 (\overline D) \cap
C^2 (D)$ be real-valued. If for every $x \in D$ there exists $r (x)$ such that $S_{r
(x)} (x)$ is admissible and equality \eqref{MM} holds with $r = r (x)$ and fixed
$\mu > 0$, then $u$ is panharmonic in $D$ and the coefficient in
equation~\eqref{MHh} is $\mu^2$.
\end{theorem}

This allows us to obtain the following analogue of Theorem 2 for panharmonic
functions.

\begin{theorem}
Let $D$ be a Dirichlet domain in $\RR^m$, $m \geq 2$, and let $u \in C^0 (\overline
D) \cap C^2 (D)$ be real-valued. If for all $r \in (0, r_*)$, where $r_*$ is a
positive number such that $D_{r_*} \neq \emptyset$, the mean $M (\cdot, r, u)$
satisfies equation \eqref{MHh} in $D_r$ and the coefficient is $\mu^2$ for all $r$,
then $u$ is panharmonic in~$D$ with the same coefficient in \eqref{MHh}.
\end{theorem}

The proof is literally the same as that of Theorem 2, but the reference to Theorem~6
must be made instead of Theorem~3.

\end{document}